%% file: monorels-siam.tex
\documentclass[review,onefignum,onetabnum]{siamart190516}

\usepackage{extarrows}
\usepackage{tikz, pgfplots}
\pgfplotsset{compat=newest}
\usepgfplotslibrary{groupplots, statistics}
\usepackage{mdframed}
\usepackage{tabularx}
\usepackage{colortbl}	
\usepackage{xspace}
\usepackage{nicefrac}
\usepackage{placeins}
\input{shared}

\input{commands}

\nolinenumbers
\ifpdf
\hypersetup{
  pdftitle={Convexification via monomial patterns},
  pdfauthor={A. Averkov, B. Peters, and S. Sager}
}
\fi

\begin{document}

\maketitle

\begin{abstract}
  Convexification is a core technique in global polynomial optimization. Currently, there are two main approaches competing in theory and practice: the approach of nonlinear programming and the approach based on positivity certificates from real algebra. 
 	The former are comparatively cheap from a computational point of view, but typically do not provide tight relaxations with respect to bounds for the original problem.
  The latter are typically computationally expensive, but do provide tight relaxations. 
  We embed both kinds of approaches into a unified framework of monomial relaxations. 
  We develop a convexification strategy that allows to trade off the quality of the bounds against computational expenses.
  Computational experiments show very encouraging results.
\end{abstract}



\section{Introduction}
\label{sec:intro}
\input{introduction.tex}

\section{Basic Notation}
\label{sec:basic}
\input{basic-notation.tex}

\section{Pattern Relaxation}
\label{sec:rlx}
\input{pattern-relaxation.tex}

\section{Known Convexification Techniques are Monomial Patterns}
\label{sec:exam-discus}
\input{exam-discus.tex}

\section{Truncated Submonoids}
\label{sec:ts-pat} 
\input{patterns.tex}

\section{Computational Results}
\label{sec:res}
\input{comp-results.tex}

\section{Conclusion}
\label{sec:conc}
\input{conclusion.tex}


\bibliographystyle{siamplain}
\bibliography{literature.bib}
\end{document}

%% file: shared.tex
\usepackage{lipsum}
\usepackage{amsfonts}
\usepackage{graphicx}
\usepackage{epstopdf}
\usepackage{algorithmic}
\ifpdf
  \DeclareGraphicsExtensions{.eps,.pdf,.png,.jpg}
\else
  \DeclareGraphicsExtensions{.eps}
\fi

\newsiamremark{hypothesis}{Hypothesis}
\crefname{hypothesis}{Hypothesis}{Hypotheses}
\crefname{proposition}{Proposition}{Propositions}
\newsiamthm{claim}{Claim}

\crefname{section}{Section}{Sections}

\newtheorem{example}{Example}

\crefname{subsection}{Subsection}{Subsections}

\pgfplotsset{
    boxplot/hide outliers/.code={
        \def\pgfplotsplothandlerboxplot@outlier{}%
    }
}

\headers{Convexification via monomial patterns}{A. Averkov, B. Peters, and S. Sager}

\title{Convexification of box-constrained polynomial optimization problems via monomial patterns \thanks{
\funding{This work was funded by the the Deut\-sche For\-schungs\-
ge\-mein\-schaft (DFG, German Research Foundation) - 314838170, GRK 2297
MathCoRe.}}}

\author{Gennadiy Averkov\thanks{Fakultät 1, Brandenburgische Technische Universität Cottbus-Senftenberg, Germany
  (\email{averkov@b-tu.de}).}
\and Benjamin Peters\thanks{Fakultät für Mathematik, Otto-von-Guericke Universität Magdeburg, Germany 
  (\email{benjamin.peters@ovgu.de}, \email{sager@ovgu.de}).}
\and Sebastian Sager\footnotemark[3]}

\usepackage{amsopn}


%% file: commands.tex
\definecolor{wheat}{rgb}{0.96,0.87,0.70}

\newcommand{\yalmip}{\texttt{YALMIP}\xspace}
\newcommand{\cstssos}{\texttt{CS-TSSOS}\xspace}
\newcommand{\baron}{\texttt{BARON}\xspace}
\newcommand{\julia}{\texttt{JULIA}\xspace}
\newcommand{\mosek}{\texttt{MOSEK}\xspace}
\newcommand{\matlab}{\texttt{MATLAB}\xspace}

\newcommand{\lvlL}{\cellcolor[gray]{0.9}}

\newcommand{\TS}{\operatorname{TS}}
\newcommand{\ML}{\operatorname{ML}}
\newcommand{\CH}{\operatorname{CH}}

\newcommand{\BF}{\operatorname{BF}}

\newcommand{\lb}{\left(}
\newcommand{\rb}{\right)}

\newcommand{\llb}{\left\{}
\newcommand{\rrb}{\right\}}
\newcommand{\set}[2]{\llb #1 : #2\rrb}
\newcommand{\bsprod}[2]{\left< #1 , #2 \right>}

\DeclareMathOperator{\conv}{conv}

\DeclareMathOperator{\diam}{diam}

\DeclareMathOperator{\md}{d}

\DeclareMathOperator{\m}{m}
\DeclareMathOperator{\supp}{supp}

\newcommand{\0}{\mathbf{0}}
\newcommand{\1}{\mathbf{1}}

\newcommand{\bfc}{\mathbf{c}}

\newcommand{\bfe}{\mathbf{e}}
\newcommand{\bff}{\mathbf{f}}

\newcommand{\bfl}{\mathbf{l}}

\newcommand{\bfu}{\mathbf{u}}
\newcommand{\bfv}{\mathbf{v}}
\newcommand{\bfw}{\mathbf{w}}
\newcommand{\bfx}{\mathbf{x}}
\newcommand{\bfy}{\mathbf{y}}
\newcommand{\bfz}{\mathbf{z}}

\newcommand{\tbff}{\tilde{\mathbf{f}}}

\newcommand{\tbfx}{\tilde{\mathbf{x}}}

\newcommand{\bfM}{\mathbf{M}}

\newcommand{\rma}{\mathrm{a}}
\newcommand{\rmA}{\mathrm{A}}
\newcommand{\rmb}{\mathrm{b}}
\newcommand{\rmB}{\mathrm{B}}
\newcommand{\rmF}{\mathrm{F}}
\newcommand{\rmc}{\mathrm{c}}
\newcommand{\rmd}{\mathrm{d}}

\newcommand{\rmf}{\mathrm{f}}
\newcommand{\rmg}{\mathrm{g}}
\newcommand{\rmh}{\mathrm{h}}
\newcommand{\rmi}{\mathrm{i}}
\newcommand{\rmI}{\mathrm{I}}
\newcommand{\rmj}{\mathrm{j}}
\newcommand{\rmJ}{\mathrm{J}}
\newcommand{\rmk}{\mathrm{k}}
\newcommand{\rmK}{\mathrm{K}}
\newcommand{\rml}{\mathrm{l}}
\newcommand{\rmm}{\mathrm{m}}
\newcommand{\rmn}{\mathrm{n}}
\newcommand{\n}{\mathrm{n}}

\newcommand{\rmP}{\mathrm{P}}

\newcommand{\rmr}{\mathrm{r}}
\newcommand{\rms}{\mathrm{s}}
\newcommand{\rmt}{\mathrm{t}}
\newcommand{\rmu}{\mathrm{u}}
\newcommand{\rmv}{\mathrm{v}}
\newcommand{\rmV}{\mathrm{V}}
\newcommand{\rmw}{\mathrm{w}}
\newcommand{\rmx}{\mathrm{x}}
\newcommand{\rmX}{\mathrm{X}}
\newcommand{\rmy}{\mathrm{y}}

\newcommand{\rec}{\mathrm{rec}}
\newcommand{\trmx}{\tilde{\mathrm{x}}}

\newcommand{\rmC}{\mathrm{C}}

\newcommand{\rmS}{\mathrm{S}}

\newcommand{\mc}{\mathrm{mc}}

\newcommand{\sgl}{\mathrm{sgl}}

\newcommand{\mF}{\mathcal{F}}

\newcommand{\mM}{\mathcal{M}}

\newcommand{\N}{\mathbb{N}}
\newcommand{\Ndd}{\mathbb{N}_{2\rmd}}
\newcommand{\Nn}{\mathbb{N}^\n}
\newcommand{\Nnd}{\mathbb{N}^\n_\rmd}

\newcommand{\Nk}{\mathbb{N}^\rmk}
\newcommand{\Nkd}{\mathbb{N}^\rmk_\rmd}
\newcommand{\Nkdd}{\mathbb{N}^\rmk_{2\rmd}}

\newcommand{\R}{\mathbb{R}}
\newcommand{\Rn}{\mathbb{R}^{\n}}
\newcommand{\RA}{\mathbb{R}^{\rmA}}

\newcommand{\Z}{\mathbb{Z}}
\newcommand{\Zn}{\mathbb{Z}^\n}

\newcommand{\Aex}{\rmA_{\mathrm{ex}}}
\let\Box\relax\DeclareMathOperator{\Box}{Box}
\newcommand{\bs}{\backslash}

\renewcommand{\epsilon}{\varepsilon}
\newcommand{\fex}{f^{\mathrm{ex}}}
\newcommand{\fa}[2][0]{\ensuremath{\hspace{#1 pt}\operatorname{for} \, \operatorname{all}\hspace{#2 pt}}}
\newcommand{\for}[2][0]{\ensuremath{\hspace{#1 pt}\operatorname{for}\hspace{#2 pt}}}

\newcommand{\minimize}[1]{\ensuremath{\operatorname{minimize}\hspace{#1 pt}}}

\newcommand{\st}[2][0]{\ensuremath{\hspace{#1 pt}\operatorname{subject} \, \operatorname{to}\hspace{#2 pt}}}

\newcommand{\xmax}[2]{\ensuremath{\operatorname{\overline{\mathbf{x}}}_{#2}^{#1}}}
\newcommand{\xmin}[2]{\ensuremath{\operatorname{\underline{\mathbf{x}}}_{#2}^{#1}}}


%% file: introduction.tex

Many important convexification techniques applied to polynomial optimization problems share the following common  distinctive features: in the case of a problem in $\n$ variables $\bfx = (\rmx_1,\ldots,\rmx_\n)$, 
monomials 
\[
	\bfx^\alpha := \rmx_1^{\alpha_1}\cdot\ldots\cdot\rmx_\n^{\alpha_\n}
\]
with $\alpha\in\Nn$ are substituted with \emph{monomial variables} $\rmv_\alpha$ and the relationships among them are captured, exactly or in a relaxed fashion, by systems of convex constraints. In order to describe the relationship between different monomial variables by constraints one needs to introduce additional \emph{auxiliary monomial variables}. 

Different approaches exist on how to pick these auxiliary monomial variables and the respective convex constraints. The ``nonlinear optimization community'' uses monomial variables and constraints such that the resulting relaxations are cheap to compute. 
The resulting poor lower bounds are compensated by solving many relaxations within a branch-and-bound framework. 
The ``polynomial optimization community'' usually aims to solve only one single relaxation, which often produces a very tight bound. This comes at the price of a large number of monomial variables and hard constraints. 
Interestingly, up to now there has been little interaction between the two different schools of thought. The authors believe that a major reason is the lack of a mathematical formalism that would allow a uniform description of different convexification techniques.


One contribution of this paper is the introduction of the notion of \emph{patterns} to fill this gap. Patterns are finite sets $\rmP\subseteq\Nn$ of exponent vectors that are chosen in such a way that the monomial variables $\rmv_{\alpha}$ indexed by $\alpha\in\rmP$ can be linked by constraints that satisfy a given demand on the computability. While various kinds of patterns have been implicitly used by the disjoint research communities, the introduction of the explicit notion of patterns allows for the development of a unifying mathematical language that highlights common ideas. Promoting the elementary notion of patterns enables to see similarities of the different research directions and will help to connect different communities that work independently on the same problems. 

For example, the pattern $\{ (1,0), (0,1), (1,1)\}$ corresponds to the well-known McCormick envelope  \cite{Mitsos2009,Boland2017}, i.e. the convexification of the variables $\rmx_1$ and $\rmx_2$ and their product $\rmx_1 \rmx_2$. Other examples of methods that can be expressed using the notion of patterns are truncated moment relaxation and its dual the sum-of-squares relaxation \cite{lasserre2011,laurent2009sums,marshall2008}, scaled-diagonally-dominant sums of squares \cite{ahmadi2019dsos}, sums of non-negative circuit polynomials \cite{dressler2017,seidler2018experimental}, bound-factor products \cite{dalkiran2013boundfactor} and their dual Handelman's hierarchy \cite{handelman1988representing}, multilinear intermediates \cite{bao2015global}, polyhedral outer approximations \cite{tawarmalani2005polyhedral} as well as expression trees \cite{smith2001symbolic,sherali2013reformulation}.
We propose a flexible template for the relaxation of \emph{box-constrained polynomial optimization problems (pop)} that allows to use the ideas of these until now largely disjoint schools of thought. 
It allows to combine different types of patterns to build convex relaxations of a pop. Our new and more general point of view might also help to understand numerical issues and the facial structures of feasible sets in the aforementioned convexification approaches. 
This, in turn, can be expected to have a positive impact on the improvement of existing and on the development of novel approaches to polynomial optimization.

We address in this paper the case of box-constrained pops.
In nonlinear global optimization, convexification of expressions occurring in constraints and objective functions (with the underlying variables in specified finite ranges) is a widely used technique. 
Since objective functions and constraints can be convexified by the same principles, one could also use our strategy for more general versions of polynomial optimization, with more general sets of constraints. 
On the other hand, developing our strategy into a sound method for general polynomial optimization would require more thought and ideas. 
Therefore, this is out of scope for this paper.
Furthermore, box-constrained subproblems are an essential part of branch-and-bound frameworks such as employed in \baron \cite{sahinidis:baron:17.8.9}.
Thus, in the future one can also try to employ our method developed for box-constrained pops with more general constraints by using them within branch-and-bound frameworks.

We derive various new convexification techniques from the monomial pattern template. The resulting relaxations can be solved by a variety of different numerical approaches. In the interest of analyzing the tightness and computational expenses related to different convexification strategies, we use the interior point solver \mosek.  

The paper is organized as follows. 
The basic notation is given in \cref{sec:basic}. 
In \cref{sec:rlx} the notion of the pattern relaxation is introduced and the separation problem for patterns is formulated as an optimization problem. \cref{sec:exam-discus} is dedicated to the interpretation and discussion of established convexification techniques as monomial patterns. 
Multilinear envelopes are generalized as multilinear patterns. 
In \cref{sec:ts-pat} new pattern types are introduced, which give rise to new algorithmic approaches to pop.  
Computational results in \cref{sec:res} highlight the benefits of our novel approach.
Finally, a conclusion is given in \cref{sec:conc}.
%

%% file: basic-notation.tex
$\N$ is the set of natural numbers including zero. 
For integers $\n>0$ and $\rmd\ge 0$ we define $\Nnd:=\set{\alpha\in\Nn}{\alpha_1+\dots+\alpha_\n\le\rmd}$, $[\rmn]:=\{1,\dots,\rmn\}$ and $[\rmn]_0:=[\rmn]\cup\{0\}$. 
Let $\rmA,\rmB\subseteq\Nn$ be nonempty, finite sets with cardinalities $\#\rmA$ and $\#\rmB$. 
We denote vectors of real numbers with entries indexed by the elements of set $A$ as $\bfv = (\rmv_{\alpha})_{\alpha\in\rmA} \in \RA$. 
Note that $\RA$ is isomorphic to $\R^{\#\rmA}$. 
We define the bilinear product of two such vectors $\bfv=(\rmv_{\alpha})_{\alpha\in\rmA}\in\RA$ and $\bfw=(\rmw_{\alpha})_{\alpha\in\rmB}\in\R^\rmB$  as
\[
  \langle\bfv,\bfw\rangle:=\sum\limits_{\alpha\in\rmA\cap\rmB}\rmv_{\alpha}\rmw_{\alpha}.
\]
Furthermore, if $\rmB\subseteq\rmA$, we define the coordinate projection of $\bfv$ onto components indexed by $\rmB$ as
$
  \bfv_{\rmB}:=(\rmv_{\alpha})_{\alpha\in\rmB}.
$
The $\ell_1$ and $\ell_{\infty}$ norms of $\bfv$ are $\Vert\bfv\Vert_{1}$ and $\Vert\bfv\Vert_{\infty}$, respectively. 
Let $\rmX\subseteq\RA$ be a nonempty and compact set. We call
\[
  \diam(\rmX):=\max\limits_{\bfu,\bfz\in\rmX} \Vert \bfu-\bfz \Vert_1,
\]
the \emph{diameter} of $\rmX$ and
\[
  \omega_{\rmX}(\bfc) :=\max\{\bsprod{\bfc}{\bfu} :\bfu\in\rmX \}-\min\{\bsprod{\bfc}{\bfu} :\bfu\in\rmX \}
\]
the \emph{width function} of $\rmX$ in direction $\bfc$.
We define the \emph{support of a vector} $\bfv$ and the \emph{support of a set} $\rmX$ as
\[
  \supp(\bfv):=\{\alpha\in\rmA: \rmv_{\alpha}\not= 0\} \quad\text{and}\quad \supp(\rmX):=\bigcup_{\bfv\in\rmX} \supp(\bfv).
\]
$\bfv$ is said to have \emph{full support} if $\supp(\bfv)= \rmA$. 
The standard basis vectors of $\RA$ are denoted by $\bfe^{\alpha}$ for $\alpha\in A$ and the all ones vector by $\1$. 
$\R[\bfx]$ is the ring of polynomials in a vector of $\rmn$ intermediates $\bfx=(\rmx_1, \dots, \rmx_\rmn)$ and $\R[\bfx]_\rmA$ the set of polynomials $f(\bfx)=\sum_{\alpha\in\rmA}\rmf_{\alpha}\bfx^{\alpha}$ . 
That is, by $\rmA$ we prescribe which monomials can occur in $f$. 
The vector $\bff=(\rmf_{\alpha})_{\alpha}$ is called the coefficient vector of $f$. 
The \emph{monomial support of a polynomial $f$} is $\supp(f)=\supp(\bff)$. 
A polynomial $p \in \R[\bfx]$ is called \emph{sum-of-squares} (sos), if $p=(p^1)^2 + \dots + (p^\rmk)^2$ for finitely many polynomials $p^1,\ldots,p^\rmk \in \R[\bfx]$. 
We use $\Sigma_{n,2d}$ to denote the cone of $\rmn$-variate sos of degree at most $2\rmd$. 
The \emph{$\rmA$-truncated moment vector map}  is
\[
  \m(\bfx)_\rmA:=(\bfx^{\alpha})_{\alpha\in\rmA}.
\]
The minimum and maximum of the monomial $\bfx^{\alpha}, \alpha\in\rmA$, over a compact set $\rmK\subseteq\Rn$ are 
\[
\xmin{\alpha}{\rmK}:=\min_{\bfx\in\rmK} \bfx^{\alpha}\quad\text{and}\quad\xmax{\alpha}{\rmK}:=\max_{\bfx\in\rmK} \bfx^{\alpha},
\]
respectively, $\xmin{A}{\rmK}:=(\xmin{\alpha}{\rmK})_{\alpha\in\rmA}$ and $\xmax{\rmA}{\rmK}:=(\xmax{\alpha}{\rmK})_{\alpha\in\rmA}.$ 
The \emph{degree of the set} $\rmA$ is
$
  \deg(\rmA):=\max\{\Vert\alpha\Vert_1 : \alpha\in\rmA\}.
$
For vectors we understand notions like $<,\le,>,\ge$ componentwise. Let $\bfl,\bfu\in\Rn$ with $\bfl < \bfu$, we define the box as 
$
  \Box(\bfl,\bfu):=[\rml_1,\rmu_1] \times \dots \times [\rml_\rmn,\rmu_\rmn]\subseteq\Rn.
$ 
We use $psd$ to abbreviate positive semidefinite.

%% file: pattern-relaxation.tex
\subsection{Monomial Convexification and Monomial Relaxation}
Let $\rmA\subset\Nn$ be a finite and nonempty set and $\bfl,\bfu\in\Rn$ be given. 
We consider the problem of minimizing a polynomial $f\in\R[\bfx]_\rmA$ over the box $\rmK:=\Box(\bfl,\bfu)$, i.e.
\begin{align} \tag{POP}\label{POP}
  \begin{array}{clll}
  \minimize{0}  & \multicolumn{2}{l}{f(\bfx)} \\
  \for[0]{0}    & \bfx  & \in\Rn\\
  \st[0]{0}     & \bfx  & \in\rmK. 
  \end{array}  
\end{align}
Via lifting, we reformulate \eqref{POP} as an optimization problem in $\RA$ with a linear objective:
\begin{align*}
  \begin{array}{cl@{\,}l@{\,}l}
  \minimize{0}  & \langle   & \multicolumn{2}{@{}l}{\bff,\bfv \, \rangle}  \\
  \for[0]{0}    &           & \bfv & \in\RA\\
  \st[0]{0}     &           & \bfv & \in \set{\m(\bfx)_\rmA}{\bfx \in\rmK}. 
  \end{array}   
\end{align*}
Replacing the feasible set by its convex hull $ \mM(\rmK)_\rmA:=\conv(\set{\m(\bfx)_\rmA}{\bfx\in\rmK})$
yields the \emph{monomial convexification of} \eqref{POP}:
\begin{align}\tag{C-POP}\label{C-POP}
  \begin{array}{cl@{\,}l@{\,}l}
  \minimize{0}  & \langle   & \multicolumn{2}{@{}l}{\bff,\bfv \, \rangle}  \\
  \for[0]{0}    &           & \bfv & \in\RA\\
  \st[0]{0}     &           & \bfv & \in\mM(\rmK)_\rmA. 
  \end{array}   
\end{align}
We refer to $\mM(\rmK)_\rmA$ as a \emph{($\rmn$-variate) moment body}. 
Clearly, the convexification \eqref{C-POP} of \eqref{POP} is tight, that is, the optimal values of \eqref{C-POP} and \eqref{POP} coincide. 
For general sets $\rmA$, the constraint $\bfv\in\mM(\rmK)_\rmA$ is difficult to verify.
Thus, it is natural to relax $\bfv\in\mM(\rmK)_\rmA$ to a system of simpler constraints of the same type
\begin{align}
  \label{pattern:relaxation}
  \bfv_{\rmP}\in \mM(\rmK)_{\rmP} \text{ for } \rmP\in\mF,
\end{align}
where $\mF$ is a finite family of finite subsets of $\Nn$ that satisfies
\begin{align}\label{pattern:cond1}
  \rmA\subseteq \bigcup_{\rmP\in\mF}\rmP. 
\end{align}
Our intention is to cover $\rmA$ by sets $\rmP\in\mF$ such that the corresponding moment bodies $\mM(\rmK)_\rmP$ yield more structure that we can exploit algorithmically than the original moment body $\mM(\rmK)_\rmA$.
We call $\rmP\in\mF$ a \emph{pattern} and \eqref{pattern:relaxation} a \emph{pattern relaxation of $\mM(\rmK)_\rmA$} with respect to the \emph{pattern family} $\mF$. 
Throughout the paper we use $\rmA_{\mF}$ to denote $\bigcup_{\rmP\in\mF}\rmP$ and refer to $\alpha\in\rmA$ as \emph{original exponents} and to $\alpha\in\rmA_{\mF}\bs\rmA$ as \emph{auxiliary exponents}. 
Using a pattern relaxation of $\mM(\rmK)_\rmA$ we obtain a lower bound on \eqref{POP} by solving
\begin{align}\tag{P-RLX}\label{P-RLX}
  \begin{array}{cl@{\,}l@{\,}l}
  \minimize{0}  & \langle   & \multicolumn{2}{@{}l}{\bff,\bfv\,\rangle}  \\
  \for[0]{0}    &           & \bfv      & \in\R^{\rmA_{\mF}}\\
  \st[0]{0}     &           & \bfv_{\rmP}  & \in\mM(\rmK)_{\rmP} \,\fa[0]{0}\, \rmP\in\mF. 
  \end{array}   
\end{align}
An advantage of this approach is that we can choose patterns $\rmP\in\mF$ such that the computational costs of solving (possibly several instances of) \eqref{P-RLX} and the obtained lower bounds on the objective function value of \eqref{POP} are well balanced.

This procedure can also be seen as embedding $\mM(\rmK)_{\rmA}$ into $\mM(\rmK)_{\rmA_{\mF}}$ for some set $\rmA_{\mF}$ that contains $\rmA$ and can be represented nicely as a union of patterns $\rmP\in\mF$. 
Geometrically, the passage from \eqref{POP} through \eqref{C-POP} to \eqref{P-RLX} can be represented by the diagram
\begin{align*}
  \m(\rmK)_\rmA \xlongrightarrow{\text{convexifying}}\mM(\rmK)_{\rmA}\xlongrightarrow{\text{embedding}}\mM(\rmK)_{\rmA_{\mF}} \xlongrightarrow{\text{projecting}}\mM(\rmK)_{\rmP}.
\end{align*}

The quality of a pattern relaxation of $\mM(\rmK)_{\rmA}$ with respect to the family of patterns $\mathcal{P}$ depends on how the moment variables are connected by the system of conditions \eqref{pattern:relaxation}. 
We say that monomial variables $\rmv_{\alpha},\bfv_{\beta}$ are \emph{directly connected} by $\mathcal{P}$ if $\alpha,\beta\in\rmP\setminus\{\0\}$ holds for some $\rmP\in\mF$. 
Furthermore, $\rmv_{\alpha},\rmv_{\beta}$ are \emph{indirectly connected} by $\mF$ if there exist $\rmP_\rmj\in\mF,\rmj \in [\rmk]$, such that $\alpha\in \rmP_{1},\beta\in \rmP_{\rmk}$ and $\rmP_{\rmj}\cap\rmP_{\rmj+1}\setminus\{\0\}\not=\emptyset$ for all $ \rmj\in [\rmk-1]$.

%% file: exam-discus.tex
We formulate established convexification techniques from the literature as monomial patterns. These pattern types can be used -- alone or in combination -- to generate computationally tractable pattern relaxations \eqref{P-RLX} of \eqref{POP}. 
\subsection{Multilinear Pattern}
Let $\rmI\subseteq\{0,1\}^\rmn$, $\rmI \not= \emptyset$ and $\alpha\in\Nn$. 
We call 
\begin{align}\label{pat:ml}
  \ML(\alpha,\rmI) 	&:=	\{(\alpha_1\omega_1,\dots,\alpha_\rmn\omega_\rmn)\in\Nn : \omega\in\rmI\}
\end{align}
a \emph{multilinear pattern} (ML), see subplot $\rmA_{1}$ in \cref{fig:multilinear-pattern} for an illustration. 
It is well known that the convex envelope of multilinear functions over $\rmK=\Box(\rml, \rmu)$ is a polytope. 
In our context this implies the following.
\begin{proposition}\label{prop:mbody-ml}
  Let $\alpha\in\Nn$ be of full support. 
  The moment body $\mM(\rmK)_{\ML(\alpha,\rmI)}$ is a polytope satisfying 
  $$\mM(\rmK)_{\ML(\alpha,\rmI)}=\conv(\m(\rmV)_{\rmI})$$ 
  with $\rmV:=\{\xmin{\alpha_1\bfe^1}{\rmK}, \xmax{\alpha_1\bfe^1}{\rmK}\}\times \dots \times \{\xmin{\alpha_\rmn\bfe^\rmn}{\rmK}, \xmax{\alpha_\rmn\bfe^\rmn}{\rmK}\}$.
\end{proposition}

Multilinear patterns can be found in different contexts in the literature. 
In their basic version they are used to convexify multilinear polynomials. An essential building block for the convexification of product terms is the McCormick envelope \cite{MR469281}, that is the convexification  of bilinear products $\rmx_1\rmx_2$ by a tight description of the moment body $\mM(\rmK)_{\ML((1,1), \{0,1\}^2)}$, noting that $\ML((1,1), \{0,1\}^2) = \{0,1\}^2$.  
McCormick envelopes have been successfully used to build convex relaxations of multilinear monomials by applying them recursively. 
For a monomial $\bfx^{\alpha}$ with $\alpha\in\{0,1\}^\rmn$ and $\#\supp(\alpha)\ge 2$ this recursion can be described as follows. 
Let $\rmJ=\{\alpha\}$ and $\mF^{\alpha}_{\rec}=\emptyset$. 
For each element $\beta\in\rmJ$ write $\beta$ as $\beta=\beta'+\beta''$ with $\beta',\beta''\in\{0,1\}^\rmn\bs \{\0\}$.  Remove $\beta$ from $\rmJ$ and add $\beta'$ to $\rmJ$ if $\#\supp(\beta')\ge 2$ respectively $\beta''$ to $\rmJ$ if $\#\supp(\beta'')\ge 2$. Add the multilinear pattern $\{\beta, \beta', \beta'', \0\}$ to $\mF^{\alpha}_{\rec}$.
 This procedure corresponds to a binary tree with root $\alpha$ and the moment body of each pattern in $\mF^{\alpha}_{\rec}$ is tightly described by a McCormick envelope.  

In general it is not clear how to favorably decompose a multilinear exponent $\beta\in\rmJ$. For the smallest nontrivial case $\#\supp(\beta)=3$ this has been investigated in \cite{MR3722434}.

Another way to convexify $\bfx^{\alpha}$ with $\alpha\in\{0,1\}^n$ and $\#\supp(\alpha)\ge 2$ is to introduce for each factor $\rmx_\rmi^{\alpha_\rmi}$ with $\alpha_\rmi\not=0$ a moment variable $\rmv_{\alpha_\rmi}$ \cite{del2017polyhedral}. This corresponds to the pattern  
\begin{align}
  \rmP^{\alpha}:=\ML(\alpha,\{\alpha\}\cup\set{\bfe^\rmi}{\rmi\in \supp(\alpha)}).
\end{align}  
For $\rmA=\{\alpha\}$ the pattern relaxation corresponding to the pattern family $\{\rmP^{\alpha}\}$ is tight, while relaxation corresponding to $\mF^{\alpha}_{\rec}$ is usually not tight for $\rmA=\{\alpha\}$. 
It is however not clear which system 
\begin{align}\label{P-ML}
	\bfv_{\rmP^{\alpha}} \in\mM(\rmK)_{\rmP^{\alpha}}  	\text{ for all } , \alpha\in\rmA	
\end{align}
or
\begin{align}\label{R-MC}
	\bfv_{\rmP} \in \mM(\rmK)_{\rmP} 	\text{ for all } \rmP\in\mF^{\alpha}_{\rec},  \alpha\in\rmA
\end{align}
yields a tighter convex relaxation of $\mM(\rmK)_{\rmA}$ for $\rmA\subseteq\{0,1\}^\rmn$ with $\#\rmA\ge 2$.
 This is due to the different choice of auxiliary variables and how the original moment variables are connected by the different pattern families. 
 In our definition \eqref{pat:ml}, the parameter $\rmI$ allows to flexibly choose auxiliary variables and thereby control the connective properties of the multilinear pattern family. 

Multilinear patterns have also been applied to general polynomials $f\in\R[\bfx]_\rmA$ with $\rmA\subseteq\Nn$ and $\rmA\bs\{\0\}\not=\emptyset$ \cite{bao2015global,glover1974converting}. 
Using the set $\Gamma:=\set{\gamma=\alpha_\rmi\bfe^\rmi}{\alpha\in\rmA, \rmi\in [\rmn]}\bs\{\0\}$, the substitution $\rmy_{\alpha_\rmi\bfe^\rmi}=\rmx^{\alpha_\rmi\bfe^\rmi}$ and $\tilde{\rmA}:=\{\beta\in\{0,1\}^{\Gamma}:\exists\alpha\in\rmA \text{ s.t. } \sum_{\gamma\in\Gamma}\beta_{\gamma}\gamma = \alpha\}$ a multilinear intermediate $\tilde{f}\in\R[\bfy]_{\tilde{\rmA}}$ of $f$ is generated. 
This corresponds to relaxing the usually non-polyhedral $\mM(\rmK)_\rmA$ with the polytope $\mM(\Box(\xmin{\Gamma}{\rmK},\xmax{\Gamma}{\rmK}))_{\tilde{\rmA}}$ and  
\begin{align}\label{G-ML}
  \begin{array}{cl@{\,}l@{\,}l}
  \minimize{0}  & \langle   & \multicolumn{2}{@{}l}{\tbff,\bfv \, \rangle}  \\
  \for[0]{0}    &           & \bfv  & \in\R^{\tilde{\rmA}}\\
  \st[0]{0}     &           & \bfv  & \in \mM(\Box(\xmin{\Gamma}{\rmK},\xmax{\Gamma}{\rmK}))_{\tilde{\rmA}}.
  \end{array}   
\end{align}  
\eqref{G-ML} is further relaxed using \eqref{R-MC} or \eqref{P-ML}. The entire process can be expressed using multilinear patterns as well. 
For example using \eqref{P-ML} to further relax \eqref{G-ML} yields the family $\set{\ML(\alpha,\{\1\}\cup\set{\bfe^\rmi}{\rmi\in [\rmn]})}{\alpha\in\rmA}$.

\begin{example}\label{rmk:pattern-plot}
We consider different exponent sets for $\rmn=2$ in the following,
\begin{align*}
\rmA_1  &:= \{(0,0), (0,3), (3,0), (3,3)\}, \\
\rmA_2  &:= \{(0,0), (1,1), (2,2), (3,3), (4,4), (5,5), (6,6)\}, \\
\rmA_3  &:= \{(0,4),(2,5),(2,8),(6,2)\}, \\
\rmA_4  &:= \{(0,2),(3,5),(6,8)\}, \\
\rmA_5  &:= \{(0,0), (0,3), (0,6), (2,0), (2,3), (4,0)\}, \\
\rmA_6  &:= \{(4,0), (4,1), (4,2), (4,3), (4,4), (4,5)\}, \\
\Aex &:= \{(0,2), (1,1), (2,3), (2,4), (4,0), (5,5)\} \label{ex_f}.
\end{align*}
The exponent sets and different patterns are visualized in Figures~\ref{fig:multilinear-pattern}, \ref{fig:dsos-circuit-pattern}, and \ref{fig:chains} as follows. 
The title of a subplot refers to the set of original exponents which are depicted by red squares. 
The auxiliary exponents are depicted by blue dots. A pattern $\rmP$ corresponds to an undirected smooth curve and all the colored points and squares that the curve passes through.
$\Aex$ will also be used in the numerical result section.
\end{example}

\begin{figure}[!h]
	\centering
	\includegraphics{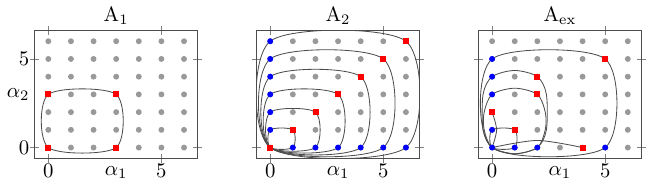}
	\caption{Visualization of multilinear patterns for example sets $\rmA_1$, $\rmA_2$, and $\Aex$ as described in \cref{rmk:pattern-plot}. 
	\textbf{Left:}  the multilinear pattern $\ML((3,3),\{0,1\}^2)=\rmA_1$;
	\textbf{Middle:}  $\mF=\{\ML(\alpha,\{0,1\}^2):\alpha\in\rmA_2\backslash\{\0\}\}$;
	\textbf{Right:}  $\mF=\{\ML(\alpha,\{0,1\}^2):\alpha\in\Aex\}$. 
	}\label{fig:multilinear-pattern}
\end{figure}

\subsection{Expression Trees}

Convexification using expression trees is common in general nonlinear optimization \cite{sherali2013reformulation,smith2001symbolic}. This approach is based on the observation that each algebraic expression is made up of a certain set of elementary operations, such as powers, linear combinations, or products of expressions.
A decomposition of an algebraic expression into these operations can be visualized using an algebraic expression tree, like in \cref{fig:symbolic-reformulation}. This is a rooted tree with nodes labeled by terms occurring in the expression. Each term is built up from its child terms using elementary operations and the underlying convexification is obtained by introducing a variable for each node and providing convex constraints that link every node and its child nodes. 
For polynomials, given as a linear combination of monomials, all the nodes apart from the root node correspond to monomial variables. A non-root node and its child nodes therefore build a pattern. 
\begin{figure}[!h]
  \centering
  \includegraphics{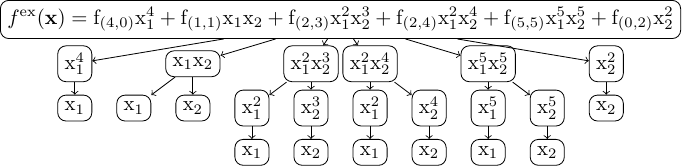}
  \caption{A possible algebraic expression tree for the polynomial $\fex$ with $\supp(\fex) \subseteq \Aex$ and the set $\Aex$ from \cref{ex_f}. }\label{fig:symbolic-reformulation}
\end{figure}
For example, the term $\rmx_1^2\rmx_2^3$ in  \cref{fig:symbolic-reformulation} is decomposed into the product of the powers $\rmx_1^2$ and $\rmx_2^3$ of the variables $\rmx_1$ and $\rmx_2$. For these three terms, one introduces the monomial variables $\rmv_{(2,3)}$, $\rmv_{(2,0)}$ and $\rmv_{(0,3)}$, respectively. The relationship of these variables is captured by the pattern $\rmP= \{ (2,3), (2,0), (0,3)\} $
and the corresponding moment body $\mM(\rmK)_{\rmP}$ is described by the well-known McCormick inequalities. The variable $\rmv_{(0,3)}$ is further  connected to  $\rmv_{(0,1)}$ by exponentiation. The corresponding pattern is $\{(0,1),(0,3)\}$. All patterns induced by the tree in \cref{fig:symbolic-reformulation} are visualized in the first subplot of \cref{fig:tree-bf-moment}.
Observe that there other ways to form expression trees. For example one could also decompose $\rmx_1^2\rmx_2^3$ into $\rmx_1\rmx_2^1$ and $\rmx_1\rmx_2^2$.
However, the corresponding pattern $\{(2,3),(1,1),(1,2)\}$ is no longer tightly described by McCormick inequalities.

Since expression trees normally correspond to patterns of small size, they lead to weak, but efficiently computable relaxations, which are often used in divide-and-conquer approaches like branch-and-bound. 

\begin{figure}[!h]
	\centering
	\includegraphics{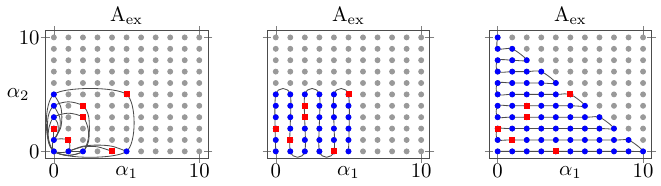}
	\caption{Visualization of the patterns corresponding to the expression tree approach, the bound-factor product and the moment relaxation   for example set $\Aex$ as described in \cref{rmk:pattern-plot}.
	\textbf{Left:}  pattern family induced by the expression tree from \cref{fig:symbolic-reformulation}, i.e., $\mF=\{\ML(\alpha,\{0,1\}^2):\alpha\in\{(1,1),(2,3),(2,4),(5,5)\} \}\cup \{ \{(0,1),(0,\rmi)\}:\rmi\in\{2,3,4,5\}\}\cup\{\{(1,0),(\rmi,0)\}:\rmi\in\{2,4,5\}\}$; 
	\textbf{Middle:} pattern corresponding to the bound-factor product $\BF((5,5))$ applied to $\Aex$; 
	\textbf{Right:} pattern of the lowest hierarchy level of the  moment relaxation $\N^2_{10}$ applied to $\Aex$.}\label{fig:tree-bf-moment}
\end{figure}

\subsection{Bound-Factor Products}

Another convexification approach is based on so-called \emph{bound-factor products (BF)} \cite{dalkiran2013boundfactor}. 
Since the polynomials $\rmx_\rmi - \rml_\rmi$ and $\rmu_\rmi-\rmx_\rmi$ are nonnegative on $\rmK$, the products of these polynomials (with repetitions allowed) are also nonnegative on $\rmK$. 
So, one can consider the products 
\begin{align}\label{ineq:boundfactor}
  F^{\alpha,\beta}(\bfx):=\prod\limits_{\rmi\in[\rmn]} (\rmx_\rmi - \rml_\rmi)^{\alpha_\rmi-\beta_\rmi}(\rmu_\rmi - \rmx_\rmi)^{\beta_\rmi}
\end{align}
of $|\alpha|$ polynomials with $\alpha_\rmi$ linear factors depending on the variable $\rmx_\rmi$, where $\alpha, \beta \in \Nn$ and $\alpha \ge \beta$.
For a generic choice of $\bfl$ and $\bfu$, the polynomial $F^{\alpha,\beta}(\bfx)$ includes all monomials with exponents in the pattern $\BF(\alpha):=\{0,\dots,\alpha_1\}\times\dots\times\{0,\dots,\alpha_\rmn\}$. 
By substituting $\rmv_{\gamma}=\bfx^{\gamma}$ for all $\gamma\in\BF(\alpha)$ we obtain a linearization $LF^{\alpha,\beta}(\bfv)$ of $F^{\alpha,\beta}(\bfx)$. The system of linear inequalities
\begin{align}\label{ineq:boundfactor-system}
  LF^{\alpha,\beta}(\bfv) \ge 0 \fa[3]{3} \beta\in \BF(\alpha)
\end{align}
is valid for $\bfv\in\mM(\rmK)_{\BF(\alpha)}$. 
This approach can also be viewed as hierarchical since one can increase the order of the bound-factor products in order to tighten the relaxation. 
Note that in polynomial optimization this approach is known as the dual of Handelman's hierarchy \cite{handelman1988representing}.  
Within this approach one groups monomial variables into patterns of a rather large size and connects them with only linear constraints.  
For example, to generate a non-trivial  relaxation of \eqref{POP} using  bound-factor products for the set $\Aex$ from \cref{ex_f} one is forced to use at least one pattern $\BF(\alpha)$ with $\alpha_1 \ge 5$ and $\alpha_2 \ge 5$, which means that at least $36$ monomial variables have to be introduced, compare \cref{fig:tree-bf-moment}. 
Another issue is that the system of linear inequalities \eqref{ineq:boundfactor-system} is not a tight description of $\mM(\rmK)_{\BF(\alpha)}$.
These kinds of relaxations have also been used within branch-and-bound strategies \cite{dalkiran2013boundfactor}.
\subsection{Moment Relaxation}
The most popular convexification techniques in the polynomial optimization community are the \emph{moment relaxation} and its dual counterpart, the \emph{sos relaxation} 
\cite{lasserre2011,laurent2009sums,marshall2008}. 
This approach introduces a large number of monomial variables and links them all with one large pattern using psd constraints. 
The approach is hierarchical in the sense that one first needs to choose a bound on the degree of the monomials, for which monomial variables are introduced. 
These hierarchies have in practice good approximation properties at the expense of large sdps, see \cite{MR1995016} for computational studies. 
Even though the lowest possible hierarchy level of the moment relaxation often produces tight bounds, it does not scale well when the number of variables and/or degree grows. 
However, strategies exist to make the approach more tractable, e.g.,  exploiting correlative sparsity \cite{MR2219151,MR3638571,MR3865829,MR3783098,mai2020sparse}, term sparsity and structures of the Newton polytope \cite{MR480338,wang2021tssos}, combinations of the previous \cite{MR4198579,wang2020cs}, symmetry structures \cite{MR3029481,lasserre2011} as well as spectral methods that exploit the so-called constant trace property of sos hierarchies \cite{mai2020hierarchy}.

To derive a so-called moment relaxation of \eqref{POP}, the following representation of the moment body $\mM(\rmK)_\rmA$ in terms of probability measures is used: 
\begin{align*}
  \mM(\rmK)_{\rmA}  &= \llb \int\m(\bfx)_\rmA \mu(d\bfx): \mu\text{ is a probability measure on } \rmK \rrb.
\end{align*}
So a vector $\bfv\in\RA$ belongs to $\mM(\rmK)_{\rmA}$ iff there exists a probability measure $\mu$ on $\rmK$ such that $\rmv_{\alpha}=\int\bfx^{\alpha}\mu(d\bfx)$ for all $\alpha\in\rmA$. Hence, \eqref{C-POP} can be formulated as 
\begin{align}\nonumber
  \begin{array}{cl@{\,}l@{\,}l}
  \minimize{0}  & \langle   & \multicolumn{2}{@{}l}{\bff,\bfv\,\rangle}  \\
  \for[0]{0}    &           & \bfv  & \in\R^{\Nn}\\
  \st[0]{0}     &           & \bfv  & \text{is a moment sequence of a probability measure on $\rmK$.}
  \end{array}   
\end{align}  
In pursuit of a tractable characterization of the feasible set, we use the following definition and theorem.
\begin{definition}[Moment Matrix and Localizing Matrix {\cite[Ch.2.7.1]{lasserre2015}}]
  The localizing matrix $\mathbf{M}_\rmk(g,\bfv)$ for a polynomial $g$ with coefficients $(\rmg_\alpha)_{\alpha}$ and the moment matrix $\mathbf{M}_\rmk(\bfv)$ are defined as 
  \begin{align*}
    \mathbf{M}_\rmk(g,\bfv) := \lb\sum\limits_{\gamma\in\Nn}\rmg_{\gamma} \rmv_{\gamma + \alpha + \beta}\rb_{\alpha,\beta\in\N^\rmn_\rmk} \quad\text{ and }\quad
    \mathbf{M}_\rmk(\bfv)                    :=\mathbf{M}_\rmk(1,\bfv).
  \end{align*}
\end{definition}
\begin{theorem}[{\cite[Th. 2.44]{lasserre2015}}]\label{theo:moment-mat-representaion}
  Let $g^1,\dots,g^\rmm$ be $\rmn$-variate polynomials such that there exist sos polynomials $s^0,\dots,s^\rmm$ for which 
  $$\{\bfx\in\Rn: s^0(\bfx )+ \sum\limits_{\rmi\in [\rmm]}s^\rmi(\bfx )g^\rmi(\bfx )\ge 0 \}$$ 
  is compact. 
  Furthermore, let $\rmK=\{\bfx\in\Rn : g^\rmi(\bfx)\ge 0 , \rmi \in [\rmm]\}.$ 
  A sequence $(\rmv_{\alpha})_{\alpha}$ has a finite Borel representing measure with support in $\rmK$ iff
  \begin{align*}
    \mathbf{M}_\rmk(\bfv)      & \text{ is psd and} \\
    \mathbf{M}_\rmk(g^\rmi,\bfv)  & \text{ is psd for all } \rmi \in [\rmm]\text{, for all } \rmk.
  \end{align*}
\end{theorem}
We describe the box $\rmK$ by the polynomials $g^\rmi(\bfx):= (\rmx_\rmi-\rml_\rmi)(\rmu_\rmi-\rmx_\rmi)$ for $\rmi\in [\rmn]$, i.e. $\rmK=\{ \bfx\in\Rn : g^\rmi(\bfx)\ge 0, \rmi \in [\rmn]\}$. 
Clearly, the assumptions of \cref{theo:moment-mat-representaion} hold and we can formulate \eqref{C-POP} as
\begin{align}\nonumber
  \begin{array}{cl@{\,}l@{\,}lll}
  \minimize{0}  & \langle   & \multicolumn{3}{@{}l}{\bff,\bfv \, \rangle}  \\
  \for[0]{0}    &           & \bfv &\in\R^{\Nn}\\
  \st[0]{0}     &           & \multicolumn{2}{@{}l}{\mathbf{M}_\rmk(\bfv)}          & \text{is psd for all } \rmk, \\
                &           & \multicolumn{2}{@{}l}{\mathbf{M}_{\rmk-2}(g^\rmi,\bfv)}  & \text{is psd for all } \rmi \in [\rmn] \text{ and all } \rmk, \\
                &           & \rmv_{\0}  &= 1.
  \end{array}   
\end{align}  

The moment and localizing matrices from the above constraints are submatrices  of infinite matrices with rows and columns indexed by $\alpha,\beta  \in \Nn$ rather than $\alpha,\beta \in \N_{\rmk}^{\rmn}$. Thus, since $\rmk$ is arbitrarily large, the constraints can be viewed as infinite-dimensional psd constraints that impose semidefiniteness of the infinite moment matrix and $\rmm$ infinite localizing matrices. By fixing a particular $\rmk = \rmd$ one relaxes the infinite dimensional psd problem to a finite-dimensional one. This is known as the choice of the level of the hierarchy of the moment relaxations. It is natural to restrict attention to levels that are sufficient large to ensure that all the variables occurring in the objective function appear in the constraints. 
Thus, for every $\rmd\ge \lceil \frac{\deg(\rmA)}{2}\rceil$, we consider the optimal value $\rho_\rmd $ of the semidefinite problem
\begin{align}\nonumber
  \begin{array}{cl@{\,}l@{\,}lll}
  \minimize{0}  & \langle   & \multicolumn{3}{@{}l}{\bff,\bfv \, \rangle}  \\
  \for[0]{0}    &           & \bfv &\in\R^{\N^\rmn_{2\rmd}}\\
  \st[0]{0}     &           & \multicolumn{2}{@{}l}{\mathbf{M}_\rmd(\bfv)}          & \text{is psd}, \\
                &           & \multicolumn{2}{@{}l}{\mathbf{M}_{\rmd-2}(g^\rmi,\bfv)}  & \text{is psd for all } \rmi \in [\rmn],  \\
                &           & \rmv_{\0}  &= 1.
  \end{array}   
\end{align}  
The value $\rho_\rmd$ is a lower bound on the optimal value of \eqref{POP}. This problem has one sdp constraint of size $\binom{\rmn+\rmd}{\rmd}$ that involves the monomial variables $\rmv_\alpha, \alpha\in\N^\rmn_{2\rmd}=\set{\alpha\in\Nn}{\alpha_1+\dots+\alpha_\rmn\le 2\rmd}$, and $\rmn$ sdp constraints of size $\binom{\rmn+\rmd-1}{\rmd-1}$ that involve $\rmv_\alpha, \alpha \in \N^\rmn_{2\rmd-2}$. 
Hence, the moment relaxation corresponds to the pattern $\N^\rmn_{2\rmd}$. 
Note that for general problems it is not possible to reduce the size of the mentioned sdp constraints \cite{averkov2018optimal}. 
For a small example like \cref{ex_f} with $\deg(\Aex)=10$ and $\rmn=2$ this adds up to 66 moment variables. 
The third subplot in \cref{fig:tree-bf-moment} shows the pattern corresponding to the lowest hierarchy level which involves an sdp constraint with a $21\times 21$ matrix. 

\subsection{Singletons}
\label{singleton:patterns}
The smallest patterns are \emph{singletons} $\{\alpha\}$ with $\alpha\in\Nn$. 
The moment body of a singleton $\{\alpha\}$ is the interval $\mM(\rmK)_{\{\alpha\}}=[\xmin{\alpha}{\rmK},\xmax{\alpha}{\rmK}]$. 
The pattern relaxation of $\mM(\rmK)_\rmA$ induced by the family of singletons $\set{\{\alpha\}}{\alpha\in\rmA}$ is $\Box(\xmin{\rmA}{\rmK},\xmax{\rmA}{\rmK})$. 
This is the weakest possible relaxation within the pattern approach. The provided bounds on the monomial variables can be exploited by branch-and-bound solvers \cite{sahinidis:baron:17.8.9}. 

\subsection{Alternative Techniques}
\begin{figure}[h!!!]
	\centering 
	\includegraphics{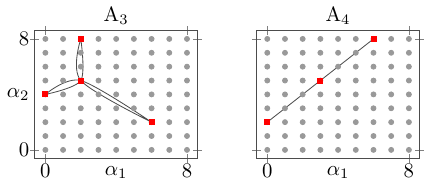}
	\caption{Visualization of the underlying patterns of sonc and sdsos for example sets $\rmA_2$ and $\rmA_4$ as described in \cref{rmk:pattern-plot}. 
	\textbf{Left:} circuit pattern (sonc) for $\rmA_3 = \{(0,4),(2,5),(2,8),(6,2)\}$; 
	\textbf{Right:} sdsos pattern for $\rmA_4 = \{(0,2),(3,5),(6,8)\}$.}\label{fig:dsos-circuit-pattern}
\end{figure}
Besides the mentioned techniques there exist other approaches for polynomial instances. For example approaches based on geometric programming or relative entropy relaxations for signomial programming have been investigated in \cite{duffin1967geometric,ghasemi2013lower,chandrasekaran2017relative,chandrasekaran2016relative}. Closely related to geometric and signomial programming are special non-negativity certificates utilizing so-called sums of nonnegative circuit polynomials (sonc)  \cite{dressler2017,seidler2018experimental,MR4007472}. 
Note that the cones of nonnegative circuit polynomials are essentially power cones. 
As it is well known that these cones have second oder cone lifts, so does the sonc cone. 
For a different proof see \cite{wang2020second}.
Furthermore, in \cite{dressler2020global} the authors purpose a linear approximation of the sonc cone.

Another approach uses scaled diagonally dominant sums of squares (sdsos) \cite{ahmadi2019dsos}, that is a non-negativity certificate based on sos polynomials with sparse monomial support $\{2\alpha, \alpha+\beta, 2\beta\}$ with $\alpha,\beta\in\Nn$.

By dualizing the sonc \cite{katthanunified} and sdsos relaxations, one arrives at convexifications in terms of monomial variables. 
These duals correspond to pattern relaxations that use special pattern types, see an illustration in \cref{fig:dsos-circuit-pattern} for the case $\rmn=2$. 

%% file: patterns.tex
In order to generate computationally tractable relaxations of \eqref{POP} we look for patterns $\rmP$ such that we can formulate the constraint $\bfv_\rmP\in\mM(\rmK)_\rmP$ of \eqref{P-RLX} (or a sufficiently tight approximation of this constraint) in such a way that it is accessible to optimization methods. 
In this section we introduce the new pattern type \emph{truncated submonoids} for which we determine the size of these constraints.

Let $\rmk\in[\rmn]$, $\rmd\in\N$  and $\Gamma=(\gamma^1, \dots, \gamma^{\rmk})\in\N^{\rmn\times\rmk}$ be a matrix, whose columns $\gamma^\rmi$ are nonzero vectors with pairwise disjoint supports. 
Clearly, such vectors $\gamma^1, \dots, \gamma^{\rmk}$  are linearly independent. 
We call 
\begin{align*}
 	\TS(\Gamma,\Nkd):=\{\gamma^1\omega_1 + \dots + \gamma^\rmk\omega_\rmk:\omega\in\Nkd\},	
\end{align*}
the \emph{$\rmk$-variate $\Nkd$-truncated submonoid} (TS) and $\gamma^1,  \dots ,\gamma^\rmk$ its \emph{generators}.
\begin{proposition}\label{03-prop:mbody-ts}
The moment body $\mM(\rmK)_{\TS(\Gamma,\Nkd)}$ can be represented as a $\rmk$-variate moment body by 
\[
  \mM(\rmK)_{\TS(\Gamma,\Nkd)} = \mM(\Box(\xmin{\Gamma}{\rmK},\xmax{\Gamma}{\rmK}))_{\Nkd}. 
\]
\end{proposition}
\begin{proof}
The desired representation is obtained by taking the convex hull of the left and the right hand side of the equality $\m(\rmK)_{\TS(\Gamma,\Nkd)} = \m(\Box(\xmin{\Gamma}{\rmK},\xmax{\Gamma}{\rmK}))_{\Nkd}$.
\end{proof}
The next proposition follows by combining \cref{theo:moment-mat-representaion} and \cref{03-prop:mbody-ts}.
\begin{proposition}\label{prop:mbody-ts-moment}
Let  $g^\rmi(\tbfx):=(\xmax{\gamma_\rmi}{\rmK} - \trmx_\rmi)(\trmx_\rmi-\xmin{\gamma_\rmi}{\rmK}) $ for each $\rmi\in [\rmk].$
Then $\bfv\in\mM(\rmK)_{\TS(\Gamma,\Nkdd)}$ if and only if there exists $\bfw\in\R^{\Nk}$ with $\rmv_{\Gamma\omega}=\rmw_{\omega}$ for all $\omega\in\Nkdd$ and
  \begin{align}\label{eq:ts-mom}
  \begin{aligned}
    \mathbf{M}_\rmr(\bfw)         & \text{ is psd and} \\
    \mathbf{M}_\rmr(g^\rmi,\bfw)  & \text{ is psd for all } \rmi \in [\rmk]\text{, for all } \rmr.
  \end{aligned}
  \end{align}
\end{proposition}

Using \cref{prop:mbody-ts-moment} we can treat the constraint $\bfv\in\mM(\rmK)_{\TS(\Gamma,\Nkd)}$ as in the moment relaxation, i.e., truncating the infinite dimensional matrices at an even $\rmr\in\N$ with $\rmr\ge\rmd$.
Naturally, the complexity of the constraints \eqref{eq:ts-mom} depends $\rmk$ and $\rmd$. 
For practical purposes, it is desirable to choose these parameters not to large. 
We use $\TS(\Gamma,\Nkdd)$, truncating the matrices in \eqref{eq:ts-mom} at $\rmr=2\rmd$, because of several reasons. 
\begin{itemize}
  \item
  Since our overall strategy for \eqref{POP}, based on \eqref{P-RLX}, does not guarantee determination of the exact optimal value of \eqref{POP}, we see no need in exact approximation of the constraints $\bfv_\rmP \in \mM(\rmK)_\rmP$ in \eqref{P-RLX} at high computational costs. 
  Therefore, when $\rmP$ is a truncated submonoid pattern, we prefer to relax $\mM(\rmK)_\rmP$ by means of \cref{prop:mbody-ts-moment} using a value $\rmr$ that is not too large.
  \item
  The lowest possible level of the moment relaxation often yields sufficiently tight bounds. 
\end{itemize}

\begin{figure}[h!!!]
  \centering
  \includegraphics{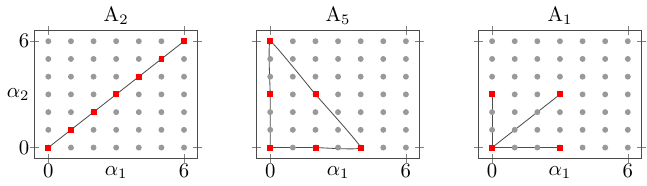}
  \caption{Visualization of different truncated submonoids for example sets $\rmA_2$, $\rmA_5$ and $\rmA_1$ as described in \cref{rmk:pattern-plot}. 
    \textbf{Left:} univariate truncated submonoid  $\TS(\1,\N_6)=\rmA_2$; 
    \textbf{Middle:} bivariate truncated submonoid $\TS((2,0),(0,3)),\N^2_2)$;
    \textbf{Right:} pattern family $\mF=\{\TS((0,3),\N_1),\TS((3,0),\N_1),\TS((3,3),\N_1)\}$ for $\rmA_1$.}\label{fig:chains}
\end{figure}	
We would like to stress that in practice the size of moment relaxations for the original problem \eqref{POP} does not scale well if the degree of $\rmA$ and $\rmn$ grows.
In general it is not possible to reduce this size if $\rmA$ does not admit any specific sparsity structures; see \cite{averkov2018optimal} for a theoretical justification. 
In contrast, we believe that one can use moment relaxations for the constraint $\bfv\in\mM(\rmK)_{\TS(\Gamma,\Nkd)}$, since we can keep the size of the matrices in \eqref{eq:ts-mom} under control. 


\subsection{Chains}\label{03-subsec:chains}
For $\gamma \in \Nn \setminus \{\0\}$ and $\rmd\in\N$, we call
\[	
	\CH(\gamma,\rmd) := \set{\rmi \gamma}{\rmi\in [\rmd]_0}
\] 
a \emph{chain}. 
A chain is a special truncated submonoid pattern with $\rmk=1$. 
In the case of chains $\rmP$, the constraints $\bfv_\rmP\in\mM(\rmK)_\rmP$ of   \eqref{P-RLX} amounts to semidefinite constraints.
\begin{theorem}[{\cite[Th. 3.23]{laurent2009sums}}]\label{03-theo:fekete}
	Let $\rmd$ be an nonnegative integer, $\rma,\rmb\in\R$ with $\rma<\rmb$ and $g(\tbfx):=(\rma - \trmx)(\trmx-\rmb).$
	Then 
	\[
	 \mM([\rma,\rmb])_{[2\rmd]_0} = \{\bfv\in\R^{\Ndd}:\bfM_{2\rmd}(\bfv) \text{ is psd and } \bfM_{2\rmd-2}(g,\bfv) \text{ is psd.}\}
	\] 
\end{theorem}
Combining \cref{03-prop:mbody-ts} and \cref{03-theo:fekete} we obtain:
\begin{proposition}\label{prop:mbody-chain-sdp}
Let $\CH(\gamma,2\rmd)$ be a chain pattern and $g(\tbfx):=(\xmax{\gamma}{K} - \trmx)(\trmx-\xmin{\gamma}{K}).$
Then the moment body $\mM(\rmK)_{\CH(\gamma,2\rmd)}$ can be represented using semidefinite constraints 
\[
  \mM(\rmK)_{\CH(\gamma,2\rmd)} = \{\bfv\in\R^{\Ndd}:\bfM_{2\rmd}(\bfv) \text{ is psd and } \bfM_{2\rmd-2}(g,\bfv) \text{ is psd}\}.
\]
\end{proposition}
\subsection{Shifting a Pattern}
To generate new patterns by shifting existing ones by a vector $\eta$, we can use the following proposition.
\begin{proposition}\label{03-prop:mbody-shifted}
  Let $\rmP\subseteq\Nn$ be a pattern with $\supp(\rmP)\not=[\rmn]$ and $\eta\in\Nn$ a vector with $\supp(\eta)\subseteq[\rmn]\backslash\supp(\rmP)$. Then
  \begin{align*}
    \mM(\rmK)_{\eta+\rmP} = \conv(\xmin{\eta}{\rmK}\mM(\rmK)_{\rmP}\cup\xmax{\eta}{\rmK}\mM(\rmK)_{\rmP}).
  \end{align*} 
\end{proposition}

\begin{proof} 
  The assertion follows from $\mM(\rmK)_{\eta+\rmP} = \conv(\{\bfx^{\eta}  (\bfx^{\beta})_{{\beta}\in\rmP}: \bfx\in\rmK\})$ and the observation that $\bfx^{\eta}$ and $\bfx^{\beta}$ have no common factor since $\supp(\eta)\cap\supp(\beta)=\emptyset$. 
  Hence
  \begin{align*}   
  \conv(\{\bfx^{\eta}  (\bfx^{\beta})_{{\beta}\in\rmP}: \bfx\in\rmK\})          
                  &=  \conv(\bigcup\limits_{\bfy\in\rmK}\bfy^{\eta} \{(\bfx^{\beta})_{{\beta}\in\rmP}: \bfx\in\rmK\}) \\
                  &=  \conv(\bigcup\limits_{\bfy\in\rmK}\bfy^{\eta}\mM(\rmK)_{\rmP})\\
                  &=  \conv(\xmin{\eta}{\rmK}\mM(\rmK)_{\rmP}\cup\xmax{\eta}{\rmK}\mM(\rmK)_{\rmP}). 
  \end{align*}
\end{proof}
\subsection{Shifted Chains}
We apply the shifting procedure to chains and generate a new pattern type. 
Let $\rmd\in\N$ and $\gamma,\eta\in\N$ with $\supp(\gamma)\cap\supp(\eta)=\emptyset$. 
We call $\eta + \CH(\gamma,\rmd)$ a \emph{shifted chain}. 
Using \cref{03-prop:mbody-shifted} we can represent the moment body $\mM(\rmK)_{\eta + \CH(\gamma,\rmd)}$ as the convex hull of  $\xmin{\eta}{\rmK}\mM(\rmK)_{\CH(\gamma,\rmd)}$ and $\xmax{\eta}{\rmK}\mM(\rmK)_{\CH(\gamma,\rmd)}$ and formulate a result for shifted chains in analogy to \cref{prop:mbody-chain-sdp}.
\begin{figure}[!h]
  \centering
  \includegraphics{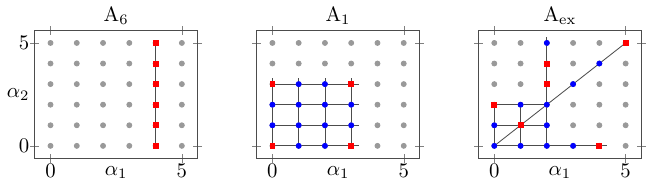}
  \caption{Visualization of different shifted chains for example sets $\rmA_6$, $\rmA_1$ and $\Aex$ as described in \cref{rmk:pattern-plot}. 
  \textbf{Left:} shifted chain $(4,0)+\CH(\bfe^2,5)=\rmA_6$;
  \textbf{Middle:} pattern family $\mF=\{(\rmi,0)+\CH(\bfe^2,3):\rmi\in [3]_0\}\cup\{(0,\rmi)+\CH(\bfe^1,3): \rmi\in [3]_0\}$ for $\rmA_1$;
  \textbf{Right:} pattern family $\mF=\{\CH(\1,5),\CH(\bfe^2,2), \CH(\bfe^1,4),(2,0)+\CH(\bfe^2,5),\bfe^2+\CH(\bfe^1,2),(0,2)+\CH(\bfe^1,2)$ for $\Aex$.}\label{fig:shifted-chains}	 
\end{figure}

\begin{corollary}
Let $\gamma,\eta\in\Nn\setminus\{\0\}$ have disjoint support and $\rmd>0$ an  integer. 
Let $g(\rmt)=(\xmax{\gamma}{\rmK}-\trmx)(\trmx-\xmin{\gamma}{\rmK})$. 
Then $\bfv\in\eta + \CH(\gamma,\rmd)$ if and only if there exists $\lambda\in[0,1]$ such that
\begin{align*}
    \lambda\bfv       &\in\{\xmin{\gamma}{\rmK}\bfw\in\R^{\Ndd}:\bfM_{2\rmd}(\bfv) \text{ is psd and } \bfM_{2\rmd-2}(g,\bfv) \text{ is psd}\},\\
    (1-\lambda)\bfv   &\in\{\xmax{\gamma}{\rmK}\bfw\in\R^{\Ndd}:\bfM_{2\rmd}(\bfv) \text{ is psd and } \bfM_{2\rmd-2}(g,\bfv) \text{ is psd}\}.
\end{align*}
\end{corollary}
\subsection{Generalizing Truncated Submonoids}
It is possible to expand the notion of truncated submonoids to generators $\Gamma=(\gamma^1, \dots, \gamma^{\rmk})\in\Z^{\rmn\times\rmk}$ by using the shifted truncated submonoids $\eta+\TS(\Gamma,\Nkdd)$ with $\eta\in 2\N$ such that 
\begin{align}\label{eq:expand-ts}
  \eta+\gamma^\rmi\ge 0\text{ for each }\rmi\in [\rmk].
\end{align}
For different choices of the parameters $\Gamma$, $\Nkdd$ and sets $\rmK$ one can apply Positivestellensätze that yield tractable characterizations of $\mM(\rmK)_{\eta+\TS(\Gamma,\Nkdd)}$.
For example let $\Gamma^1\in\Zn\bs 2\Zn, \Gamma^2\in\Z^{\rmn\times 2}\bs 2\Z^{\rmn\times 2}$ and $\Gamma^\rmk\in\Z^{\rmn\times\rmk}\bs 2\Z^{\rmn\times\rmk}$ be matrices whose columns have pairwise disjoint support and satisfy \eqref{eq:expand-ts}.
Then from \cite[Hilbert 1888]{marshall2008} it follows that $\mM(\Rn)_{\eta+\TS(\Gamma^1,\N_{2\rmd})}$, $\mM(\Rn)_{\eta+\TS(\Gamma^2,\N^2_{4})}$ and $\mM(\Rn)_{\eta+\TS(\Gamma^\rmk,\N^\rmk_{2})}$ can be represented by psd constraints of size  $\tbinom{1+\rmd}{\rmd}$, $\tbinom{2+2}{2}$ and $\tbinom{\rmk+2}{1}$, respectively.  
Furthermore, if $\Gamma^1\in 2\Zn$ satisfies \eqref{eq:expand-ts}, then $\mM(\Rn)_{\eta+\TS(\Gamma^1,\N_{2\rmd})}$ can be characterized by two psd matrices, one of size $\binom{1+\rmd}{\rmd}$, one of size $\binom{1+\rmd-1}{\rmd-1}$.
In particular, combining these representations of $\Gamma^1\in 2\Zn$ and $\Gamma^1\in\Zn\bs 2\Zn$ leads to a generalization of the underlying pattern of the sdsos certificate.
At last, if $\Gamma^1\in \Zn$ satisfies \eqref{eq:expand-ts}, then $\mM(\Box(\bfl,\bfu))_{\eta+\TS(\Gamma^1,\N_{2\rmd})}$ can be characterized using at most two psd matrices of size at most $\binom{1+\rmd}{\rmd}$. 
For that one has to determine whether the closure of $\{\m(\bfx)_{\Gamma^1}\in\R:\bfx\in\Box(\bfl,\bfu)\bs\{\0\})\}$ is $\R$, a semi-infinite interval, the union of two disjoint semi-infinite intervals or a bounded interval and then apply the respective Positivestellensatz: \cite[Hilbert 1888]{marshall2008}, \cite[Stieltjes 1885]{marshall2008}, \cite[Hausdorff 1921]{marshall2008} or \cite[Svecov 1885]{marshall2008}.
However, since we do not use any of these representations in the computations section, we do not pursue these patterns any further.

%% file: comp-results.tex
Finding an unbiased setting to compare the advantages and disadvantages of convex relaxations for pop is not trivial, as models, their purpose, and methods are usually closely linked to one another. 
We decided to use a prototype implementation to compute solutions of \eqref{P-RLX} for different monomial pattern families. 
The solutions are used to approximate the size of the relaxations, and compared on a new benchmark library of random pop instances among another and to results from \baron, \yalmip and \cstssos.
We start by describing implementation and comparison details, before numerical results for different classes of instances are discussed.

\subsection{Implementation Details}

Four different solvers were run for the numerical evaluation on a compute server with 4 Intel(R) Xeon(R) Gold 6138 CPUs with 20 cores of 2 threads and 1 TB RAM each under Ubuntu 18.04.4.
Each solver-instance pair was assigned to one such job, i.e. the solvers themselves did not use the parallel structure.
In order to distribute the solver-instance pairs to the 80 cores we used \cite{tange_2020_4118697}.
We used \matlab~9.6.0.1174912 (R2019a) Update 5 \cite{matlab}, \mosek 9.2.32 \cite{mosek}, \julia 1.5.2 \cite{julia}, \cstssos version 1.00 \cite{MR4198579}, \baron~1.8.9 \cite{tawarmalani2005polyhedral}, and \yalmip 20200930 \cite{yalmip}.
All reported run times are wall-clock times. 
The code for solving the pattern relaxation \eqref{P-RLX} was implemented and run in \matlab and consists of roughly 3500 lines of code and uses \mosek to solve the relaxations.
The reported time is the termination time obtained from \mosek.
\baron \cite{tawarmalani2005polyhedral} was called from \matlab with default settings. 
\baron currently only returns the CPU time, when its \matlab interface is used.
Hence, we timed \baron calls with \texttt{MATLAB's} \texttt{tic} and \texttt{toc} commands\footnote{This method was suggested with the support of \baron.}. 
\cstssos is a \julia package that allows to exploit correlative sparsity and term sparsity simultaneously.
We called the first level of the hierarchy by running the command \texttt{cs\_tssos\_first} with settings \texttt{order}  
$=\lceil\tfrac{\deg(f)}{2}\rceil$ and \texttt{TS="MD"}. 
\cstssos does not report the time of the solution process.
Thus, we first piped the output from the sdp solver \mosek, that \cstssos uses, to a text file.
After that we read the termination time of \mosek from the text file.
Since only two decimal places are obtained this way, the time we report is only a proxy of the time of \texttt{CS-TSSOS's} actual \mosek call.
\yalmip is a \matlab toolbox that allows to compute the moment as well as the sos relaxation of \eqref{POP}. 
We run \texttt{YALMIP's solvesos} lowest possible  level of the sos hierarchy and  report the termination time obtained from \mosek.

\subsection{Setup of Numerical Comparisons}\label{subsec:setup}

As an indicator for the tightness of relaxations we approximate the size of feasible sets by their width.
For a given finite and nonempty set $\rmA\subseteq\Nn$ and a vector $\bff\in\RA$ we define the width function $\omega_{\mM(\rmK)_{\rmA}}(\bff)$ of $\mM(\rmK)_{\rmA}$ in direction $\bff$ as
\begin{align}\label{width-poly}
  \omega_{\mM(\rmK)_{\rmA}}(\bff) = \max_{\bfx\in\rmK} f(\bfx) - \min_{\bfx\in\rmK} f(\bfx).
\end{align}
Replacing $\rmK$ by a relaxation based on a pattern family $\mF$ one obtains an upper bound on the value of $\omega_{\mM(\rmK)_{\rmA}}(\bff)$, denoted by $\omega(\mF,\mM(\rmK)_{\rmA},\bff)$. 
The evaluation requires solving two instances of \eqref{P-RLX} for every pattern of interest, using the objective functions $\bsprod{-\bff}{\bfv}$ and $\bsprod{\bff}{\bfv}$, respectively.
To normalize the values $\omega_{\mM(\rmK)_{\rmA}}(\bff)$ and  $\omega(\mF,\mM(\rmK)_{\rmA},\bff)$
, we divide by the width function obtained for the (trivial) relaxation using the singletons-only pattern $\mF^{\sgl}_\rmA = \{\{\alpha\}:\alpha\in\rmA\}$, i.e., 
\begin{align}\label{width-normalize}
 \nu(\mF,\rmA,\bff) :=  \frac{\omega(\mF,\mM(\rmK)_{\rmA},\bff)}{\omega(\mF^{\sgl}_\rmA,\mM(\rmK)_{\rmA},\bff)} \quad\text{ and }\quad \nu_\rmA(\bff)    := \frac{\omega_{\mM(\rmK)_{\rmA}}(\bff)}{\omega(\mF^{\sgl}_\rmA,\mM(\rmK)_{\rmA},\bff)}.
\end{align}

Table~\ref{tab_methods} lists the methods and patterns that were used for the numerical results.
Method (B) gives an approximation of the reference solution, albeit at a high computational cost. (R) can be seen as the current state-of-the-art for a relaxation within a divide-and-conquer approach. Our approach allows to compare the new relaxation strategies (M), (C), (MC), (H), (T) with respect to the width function. 

\cref{boxplot:adversary,boxplot:dense,boxplot:sparse,boxplot:sparse-d4,boxplot:custom} show box plots of our numerical findings.   
  The box plots visualize the distributions (20 random vectors $\bff^\rmi$) of the normalized width functions \eqref{width-normalize} for various methods from \cref{tab_methods} computed with \baron, \yalmip and \cstssos.
  The title of a subplot corresponds to the exponent set $\rmA$. 
  Below the method (see \cref{tab_methods}) the rounded mean time in seconds is shown for the respective method.
  The box borders are the $\nicefrac{1}{4}$ and the $\nicefrac{3}{4}$-quantiles. 
  The lower whisker is the smallest data value which is larger than the lower quartile $-1.5$ times the interquartile range and the upper whisker accordingly.

\begin{table}[hb!!!] 
\begin{tabularx}{\textwidth}{|c|p{11.18cm}|}\hline
\lvlL\textbf{Label} & \lvlL\textbf{Description} \\ \hline
(B)   & Reference solution: To approximate $\omega_{\mM(\rmK)_{\rmA}}(\bff)$ we use the best upper bound for $\max_{\bfx\in\rmK} f(\bfx)$ and the best lower bound for $\min_{\bfx\in\rmK} f(\bfx)$ that \baron returns within a CPU time limit of $1000$ seconds each.  \\
(R)   & Root node relaxation of the \baron solver. \\
(CS)  & Reference solution obtained from solver \cstssos. \\
(Y)   & Reference solution obtained from \texttt{YALMIP's} sos method. \\
(SOS) & Self-implemented sos relaxation (that does not exploit sparsity) of the lowest hierarchy level. \\
(M)   & Relaxation based on the multilinear patterns $\mF^{\rmm}_\rmA$, which consists of the inclusion-maximal elements of $\set{\ML(\alpha,\{0,1\}^\rmn)}{\alpha\in\rmA\bs\{0\}}.$ \\
(S)   & Relaxation based on a family of shifted chains $\mF^{\rms}_\rmA$, which consists of the inclusion-maximal elements of 
        $$\set{\eta+\CH(\bfe^\rmi,\rmd)}{\rmd\in 2\N\bs\{0\}, \eta\in\Nn, \rmi\in [\rmn]},$$ 
    that satisfy  $\#(\eta+\CH(\gamma,\rmd))\cap\rmA\ge 2$ and $$\#(\eta+\CH(\gamma,\rmd))\cap\rmA > \#(\eta+\CH(\gamma,\rmd-1))\cap\rmA.$$ 
    The latter conditions ensure that each shifted chains contains at least two exponents from $\rmA$ and that we cannot include more exponents from $\rmA$ if we choose a bigger $\rmd$. \\
(C)   & Relaxation based on a family of chains $\mF^{\rmc}_\rmA$, which consists of the inclusion-maximal elements of 
        $$\set{\CH(\gamma,\rmd)}{\rmd\in 2\N\bs\{0\}, \gamma\in\Nn},$$ 
    that satisfy  $\#\CH(\gamma,\rmd)\cap\rmA\ge 2$ and $\#\CH(\gamma,\rmd)\cap\rmA > \#\CH(\gamma,\rmd-1)\cap\rmA.$  \\
(MC)  & Relaxation based on a family of multilinear patterns, chains and shifted chains,
        $$\mF^{\mc}_\rmA:= \mF^{\rmm}_\rmA \cup \mF^{\rmc}_{\bar{\rmA}} \cup  \mF^{\rms}_{\tilde{\rmA}}$$ 
        where $\bar{\rmA}:=\rmA_{\mF^{\rmm}_\rmA}$, $\tilde{\rmA}:=\rmA_{\mF^{\rmm}_\rmA \cup \mF^{\rmc}_{\bar{\rmA}}}$.\\
(H)   & Let $\md(\rmA):=\max(\set{\alpha_\rmi}{\alpha\in\rmA,\rmi\in [\rmn]})$ and $\Gamma:=\{\1,\bfe^1, \dots, \bfe^\rmn\}$. 
        A relaxation based on the family 
        $$\mF^{\rmh}_\rmA: = \set{\CH(\gamma,\md(\rmA))}{\gamma\in\Gamma}\cup\mF^{\rmm}_{ \set{\CH(\gamma,\md(\rmA))}{\gamma\in\Gamma}}\cup \mF^{\rmm}_{\rmA},$$ 
        which uses $\rmn+1$ chains that are linked by $\md(\rmA)$ multilinear patterns to strengthen $\mF^{\rmm}_\rmA$. \\
(T)   & Let $\rmd_1:=2\cdot\lceil\nicefrac{\deg(\rmA)}{2}\rceil$, $\rmd_2:=2\cdot\lceil\nicefrac{\deg(\rmA)}{4}\rceil $ and $\Gamma:=(2\bfe^1, \dots, 2\bfe^\rmn)$. 
        A relaxation based on the family $\mF^{\rmt}_\rmA$, which consists of the inclusion-maximal elements of 
        $$\set{\TS((\bfe^\rmi)_{\rmi\in\supp(\alpha)},\rmd_1)}{\alpha\in\rmA\bs \TS(\Gamma,\rmd_2) }\cup\{\TS(\Gamma,\rmd_2)\}.$$
        Here, $(\bfe^\rmi)_{\rmi\in\supp(\alpha)}$ is a matrix with columns $\bfe^\rmi, \rmi\in\supp(\alpha)$. 
        The family $\mF^{\rmt}_\rmA$ uses $\rmk$-variate truncated submonoids with $\rmk\le\rmd_1$ to cover the exponents in $\rmA$ and connects these chains using one $\rmn$-variate truncated submonoid. \\    
\hline
\end{tabularx}
\caption{Methods and relaxations to be compared in \cref{boxplot:adversary,boxplot:dense,boxplot:sparse,boxplot:sparse-d4,boxplot:custom}. 
 \label{tab_methods}}
\end{table}
\FloatBarrier

\subsection{Test Instances}

In our test instances, we use 13 finite exponents sets $\rmA \subseteq \Nn $ classified into four types: specially structured adversary sets, dense sets, sparse sets, and the example $\Aex$ from above. 
They are explained in the next subsection. For each exponent set we chose $\rmK=[0,1]^\rmn$ and 20 (uniform distributed) random coefficient vectors $\bff^1,\dots,\bff^{20}\in [-1,1]^\rmA$. 
The instances were a priori filtered to avoid trivial problems. 
If \baron did terminate on both the minimization and the maximization tasks in \eqref{width-poly} within the CPU time limit of 1000 seconds, the instance was replaced. 
Therefore the corresponding mean times for (B) are always at least $1000$ seconds. 

Our approach to generation of test instances is a search for instances that are interesting and realistic enough, on the one hand, but computationally challenging for the existing methods, on the other hand. That is, we wanted to study if existing convexification strategies can be improved on some interesting families of optimization problems. 

While we test our approach for \eqref{POP} on the unit box $\rmK=[0,1]^\rmn$ and with the objective functions having coefficients in $[-1,1]$, our approach is applicable without any changes for general objective functions and on arbitrary axis-aligned boxes. We also expect that the results of our numerical evaluations would be the same in this slightly more general setting.

%



\subsection{Numerical Results}\label{subsec:results}
In this subsection we describe the different exponent sets and present numerical results for the different methods from Table~\ref{tab_methods}.

\subsubsection{Adversary Exponent Sets}\label{subsec:adversary}
If a pattern family yields poor connectivity properties for an exponent set, we consider this set to be an adversary exponent set for this family. 
In subplot $\rmA_2$ from \cref{fig:multilinear-pattern}, for example, we see that the sparse family $\mF^{\rmm}_{\CH(\1^5,\rmd)}$ of multilinear patterns connects none of the original exponents. 
Hence, chain shaped exponent sets are natural adversaries for relaxations that only use multilinear patterns. 
As a result, the first two subplots in \cref{boxplot:adversary} show that the bounds using $\mF^{\rmm}_{\CH(\gamma,\rmd)}$ (M) coincide with the bounds obtained by the weakest pattern family $\mF^{\sgl}_{\CH(\gamma,\rmd)}$. 
On the other hand, it is not surprising that the bounds obtained by using one chain (C) match the reference solution (B).
The sparsity exploiting solver \cstssos as well as \texttt{YALMIP's} sos method fail to terminate for any of the 20 instances with exponent set $\CH(\1^4,10)$.
We suspect that the reason for this is that $\CH(\1^4,10)$ does not yield any term or chordal sparsity structures that can be exploited.
Thus, \cstssos and \yalmip  solve a regular sos relaxation for $\rmn=4$ and $\deg(\CH(\1^4,10))=40$, involving an sdp with a $\tbinom{24}{4}\times\tbinom{24}{4} = 10626\times 10626$ psd matrix. 

Another adversary exponent set for multilinear patterns is $\rmC(\rmn,\rmd):=\CH(\1^\rmn,\rmd)\cup\CH(\bfe^1,\rmd)\cup\dots\cup\CH(\bfe^\rmn,\rmd)$. It can be covered sparsely by $d$ multilinear patterns using the family $\mF^{\rmm}_{\rmC(\rmn,\rmd)}$. 
Each pattern of $\mF^{\rmm}_{\rmC(\rmn,\rmd)}$ connects $\rmn+1$ original exponents, but establishes no connection between monomials from different patterns. 
That is because two patterns $\rmP,\rmP'\in\mF^{\rmm}_{\rmC(\rmn,\rmd)}$ with  $\rmP\not=\rmP'$ satisfy $\rmP\cap \rmP' = \{\0\}$. The poor connective properties of $\mF^{\rmm}_{\rmC(\rmn,\rmd)}$ explain their poor performance, see (M) in \cref{boxplot:adversary}. 
By additionally using $\rmn+1$ chains to connect the $\rmd$ multilinear patterns, the family $\mF^{\rmh}_{\rmC(\rmn,\rmd)}$ exploits the structure of $\rmC(\rmn,\rmd)$. 
As a result, the resulting bounds of (H) and (B) are indistinguishable in \cref{boxplot:adversary}.
Again, \cstssos and \yalmip fail to terminate for any of the instances with exponent set $\rmC(4,10)$ -- most likely for the same reason as above.
\begin{figure}[htbp]
  \centering
  \includegraphics{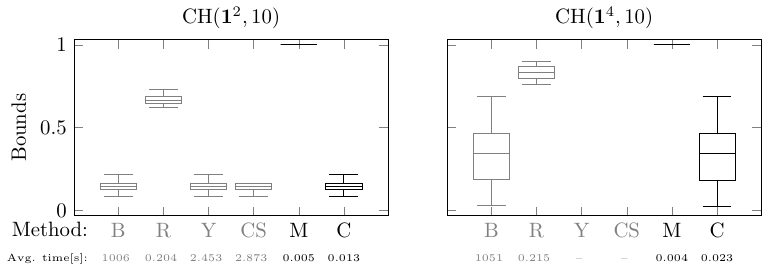}
  \includegraphics{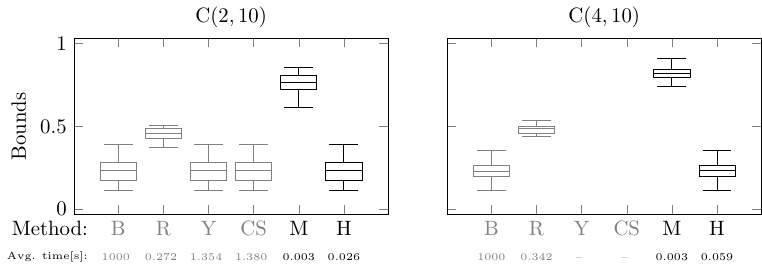}
  \caption{Visualization of the distributions (20 random vectors $\bff^\rmi$) of the normalized width functions for various methods from Table~\ref{tab_methods} as described in \cref{subsec:setup} for adversary exponent sets.
}\label{boxplot:adversary}
\end{figure}

\subsubsection{Dense Exponent Sets}\label{subsec:dense} 
We consider dense exponent sets $\rmA=\Nnd$ for $\rmn\in\{2,4\}$ and $\rmd=10$. 
The pattern families shown in \cref{boxplot:dense} perform reasonably well, probably due to their connectivity properties. Furthermore, we see that the multilinear patterns (M) perform for $\rmn=4$ drastically better than (R). This might be because the multilinear patterns $\ML(\alpha,\{0,1\}^\rmn)$ used in $\mF^{\rmm}_{\N_{10}^4}$ are bigger than the ones \baron uses, leading to more connections between monomial variables. 
\begin{figure}[h!!!]
  \centering
  \includegraphics{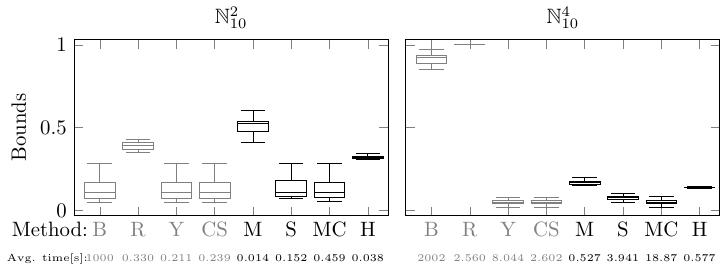}
  \caption{Visualization of the distributions (20 random vectors $\bff^\rmi$) of the normalized width functions for various methods from Table~\ref{tab_methods} as described in \cref{subsec:setup} for dense exponent sets.}\label{boxplot:dense}
\end{figure}

\subsubsection{Sparse Exponent Sets}\label{subsec:num-res-sparse}
We use randomly generated sparse exponent sets $\rmA=\rmS(\rmn,\rmd)$ to test pattern families that do not assume any structure of $\rmA$.
$\rmS(\rmn,\rmd)$ is generated by randomly picking $\left\lceil\sqrt{\tbinom{\rmn+\rmd}{\rmd}}\right\rceil$ exponents via \texttt{randperm} from $\N^{\rmn}_{\rmd}$.

\begin{figure}[h!!!]
  \centering
  \includegraphics{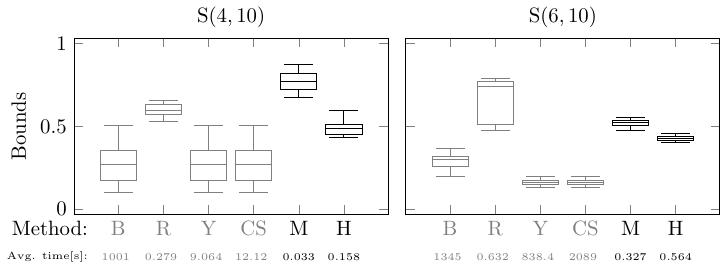}
  \caption{Visualization of the distributions (20 random vectors $\bff^\rmi$) of the normalized width functions for various methods from Table~\ref{tab_methods} as described in \cref{subsec:setup} for sparse exponent sets of degree at most 10.}\label{boxplot:sparse}
\end{figure}
\cref{boxplot:sparse} column (M) shows that $\mF^{\rmm}_{\rmS(\rmn,\rmd)}$ does not perform particularly well. 
Column (H) shows that additionally enforcing indirect connections between moment variables via $\rmn+1$ chains and $\rmd$ multilinear patterns in $\mF^{\rmh}_{\rmS(\rmn,\rmd)}$ results in tighter bounds. 

\begin{figure}[h!!!]
  \centering
  \includegraphics{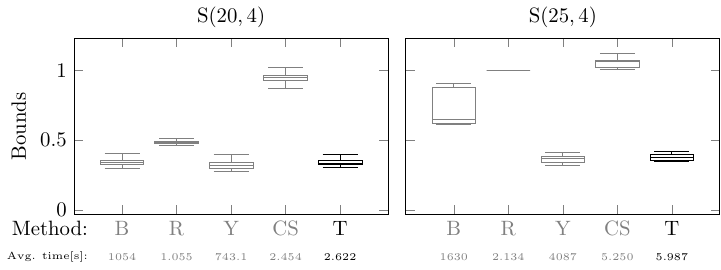}
  \includegraphics{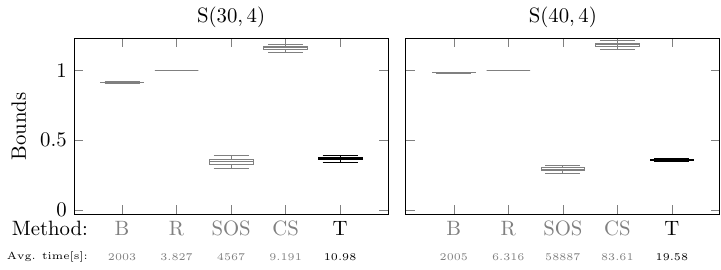}
  \caption{Visualization of the distributions (20 random vectors $\bff^\rmi$) of the normalized width functions for various methods from Table~\ref{tab_methods} as described in \cref{subsec:setup} for sparse exponent sets of degree at most 4.}\label{boxplot:sparse-d4}
\end{figure}
\cref{boxplot:sparse-d4} shows the distribution of the width for sparse instances with a high number of variables $\rmn=20,25,30,40$ and low degree $\rmd=4$. 
Computing lower bounds for the instances using relaxations that do not exploit sparsity of  $\rmS(\rmn,\rmd)$ involves severe computational cost.
We ran into memory problems with \yalmip for the instance with $\rmS(\rmn,4)$ when $\rmn\ge 35$.  
Thus, we used  (SOS), that is an own implementation  of an sos relaxation instead.
(The method (SOS) yields similar bounds to (Y) for $\rmn=20,25$ with average times $147.4$s and $935.8$s). 
Interestingly, the bounds computed by \cstssos are worse than the ones ones computed with the pattern family $\mF^{\sgl}_{\rmS(\rmn,4)}$.
It might be that using different settings for \cstssos yields better bounds.
However, this would also result in higher computation times.
The pattern strategy (T) yields for all tested $\rmn$ nontrivial bounds. 
Note that for $\rmn=20,25,30,40$ these bounds seem to be reasonably tight, when compared to (Y) or (SOS), but for a fraction of the computation time.

We want to point out that we were able to compute nontrivial bounds for instances with exponent sets $\rmS(80,4)$.
For these exponent sets the computation of one of the two optima involved in the definition of the width usually takes between 6-7 minutes.
The reason for the good performance in terms of computation time of (T) can be traced back to the relatively small size of the biggest involved $\rmr\times\rmr$ psd matrices in the relaxation of \eqref{P-RLX}. 
That is for $\mF^{\rmt}_{\rmS(\rmn,2\rmd)}= \set{\TS(\{\bfe^\rmi\}_{\rmi\in\supp(\alpha)},2\rmd)}{\alpha\in\rmS(\rmn,2\rmd)\bs \TS(\Gamma,2\rmd) }\cup\{\TS(\Gamma,\rmd)\}$ 
\[
  \rmr\le\max\Big\{\binom{\rmd+\min\{\rmn,2\rmd\}}{\min\{\rmn,2\rmd\}},\binom{\big\lceil\tfrac{\rmd}{2}\big\rceil+\rmn}{\rmn}\Big\}.
\]
For $2\rmd=4$ and $\rmn\ge 4$ this boils down to $\rmr\le\max\Big\{\tbinom{2+4}{4},\tbinom{1+\rmn}{\rmn}\Big\}$.
\subsubsection{Custom Strategies}\label{07-subsec:custom}
A customized pattern family $\mF$ for a given exponent set $\rmA$ allows to trade off computational cost versus tightness of the relaxation. 
Figure~\ref{fig:intro-pattern-configs} shows three example pattern families customized for $\Aex$ from \cref{ex_f}.
\begin{figure}[htbp]
  \centering
  \includegraphics{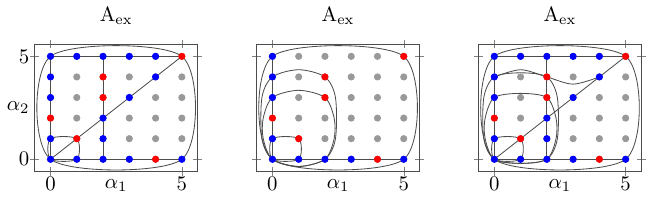}
  \caption{Visualization of different pattern families for the example set $\Aex$  as described in \cref{rmk:pattern-plot}.
  \textbf{Left:} the pattern family $\mF^1=\{\ML(\alpha,\{0,1\}^2):\alpha\in\{(1,1),(5,5)\}\}\cup\{\CH(\alpha,5): \alpha\in\{\bfe^1,\bfe^2,\1^2\}\}\cup\{(0,5) + \CH(\bfe^1,5),(2,0) + \CH(\bfe^2,5)\}$;
  \textbf{Middle:} the pattern family $\mF^2=\{\ML(\alpha,\{0,1\}^2):\alpha\in\{(1,1), (2,3), (2,4), (5,5)\}\}\cup\{\CH(\alpha,5): \alpha\in\{\bfe^1,\bfe^2\}\}$; 
  \textbf{Right:} the pattern family $\mF^3=\{\ML(\alpha,\{0,1\}^2):\alpha\in\{(1,1), (2,3), (2,4), (5,5)\}\}\cup\{\CH(\alpha,5): \alpha\in\{\bfe^1,\bfe^2,\1^2\}\}\cup \{(0,5) + \CH(\bfe^1,5), (2,0) + \CH(\bfe^2,5), (0,4) + \CH(2\bfe^1,2)\}$.
  }\label{fig:intro-pattern-configs}
\end{figure}
While the bounds obtained from $\mF^{2}$, see $\rmF^2$ in \cref{boxplot:custom}, are far from optimal, they are an improvement compared to $\mF_{\Aex}^{\rmm}$ in producing bounds similar to those obtained by (B).
\begin{figure}[h!!!]
  \centering
  \includegraphics{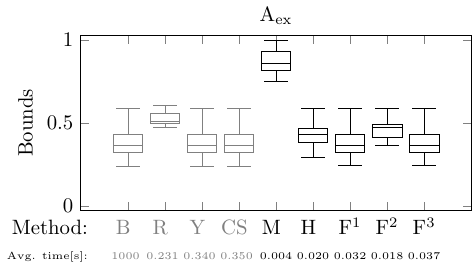}

  \caption{Visualization of the distributions (20 random vectors $\bff^\rmi$) of the normalized width functions for various methods from Table~\ref{tab_methods} as described in \cref{subsec:setup} for the set $\Aex$. The custom pattern families $\mF^1, \mF^2, \mF^3$ for $\Aex$ (labeled as $\rmF^\rmi$) are defined in \cref{fig:intro-pattern-configs}.}\label{boxplot:custom}
\end{figure}

%% file: conclusion.tex
We have presented a customizable framework for the relaxation of polynomial optimization problems over a box that is based on monomial patterns.
This framework allows inclusion and combination of existing approaches that were developed by different communities. 
In fact, various kinds of linearizations of multilinear terms, relaxations based on bound-factor products, dual versions of the relaxations of polynomial optimization problems based on sos, sdsos and sonc polynomials \emph{all} come with their particular type of pattern. 
The advantage of our approach is that by using patterns we can exploit the combinatorial structure of the set $\rmA$ of monomial exponents.
This is done by covering the monomial support $\rmA$ with a pattern family that reflects the structure of $\rmA$. 
Using patterns, we are able to avoid hard problem formulations by neglecting dependencies between certain monomials and instead focus on well-behaved and easy-to-describe dependencies between certain other monomials.
The results were high-quality and tractable relaxations of \eqref{POP}.

Our computational experiments provided numerical evidence for the benefits of using different generic as well as customized pattern relaxations.

These computed bounds could be further improved by techniques within divide-and-conquer frameworks such as \baron \cite{tawarmalani2005polyhedral}, \texttt{SCIP} \cite{scip}, \texttt{COUENNE} \cite{belotti2015couenne} or \texttt{LINDOGlobal} \cite{schrage2006optimization}, in a similar manner as already done with McCormick envelopes in the global optimization community. 

In particular, the more involved $\rmk$-variate truncated submonoids and its possible generalizations provide a way to use sos or moment methods to solve problems with polynomials of higher degree and with more variables. 
Choosing an appropriate set of generators of a truncated submonoid pattern, this pattern type could be used as an interface to combine sos methods with divide-and-conquer frameworks.
Furthermore, combining truncated submonoids with sparsity exploiting approaches such as chordal or term sparsity pose a way to further improve the run times. 

The numerical results also suggest that the connectivity properties of a pattern family have a major impact on the quality of the computed bounds.
This could be further investigated and exploited with hypergraph based approaches in the future.

How to efficiently generalize the approach to polynomial inequalities and to identify properties of instances in specific application areas that might benefit particularly from the new approach, are further open research questions.
